\documentclass{article}
\usepackage{amsfonts}
\usepackage{amssymb}
\usepackage{enumerate}
\usepackage{amsmath}
\usepackage[T1]{fontenc}
\usepackage{graphicx}
\usepackage[latin2]{inputenc}
\usepackage{amsthm}

\numberwithin{equation}{section}

\newtheorem{theorem}{Theorem}[section]
\newtheorem{lemma}[theorem]{Lemma}

\newtheorem{question}[theorem]{Question}

\theoremstyle{definition}

\newtheorem{remark}[theorem]{Remark}
\newtheorem{definition}[theorem]{Definition}

\newcommand{\NN}{\mathbb{N}}
\newcommand{\ZZ}{\mathbb{Z}}

\newcommand{\RR}{\mathbb{R}}

\renewcommand{\SS}{\mathbb{S}}

\DeclareMathOperator{\pr}{pr}

\DeclareMathOperator{\graph}{graph}

\DeclareMathOperator{\length}{length}

\newcommand{\iG}{\mathcal{G}}

\title{Continuous horizontally rigid functions of two variables are affine}

\author{Rich\'ard Balka\thanks{Partially supported by the
Hungarian Scientific Foundation grant no.~72655.}\\
Alfr\'ed R\'enyi Institute of Mathematics\\
PO Box 127, 1364 Budapest, Hungary\\
and \\
Eszterh\'azy K\'aroly College\\
Institute of Mathematics and Informatics\\
Le\'anyka u. 4., 3300 Eger, Hungary \\
email: balkar@cs.elte.hu \\
\\
M\'arton Elekes\thanks{Partially supported by the Hungarian
Scientific Foundation grants no.~72655, 61600, 83726 and J\'anos
Bolyai Fellowship.
\newline\indent {\it 2000 Mathematics Subject
Classification:} Primary 26A99 Secondary 39B22, 39B52, 39B72, 51M99.
\newline\indent {\it Keywords:} rigid, functional equation,
transformation.}\\
Alfr\'ed R\'enyi Institute of Mathematics\\
PO Box 127, 1364 Budapest, Hungary\\
and\\
E\"otv\"os Lor\'and University\\
Institute of Mathematics\\
P\'azm\'any P\'eter s. 1/c, 1117
Budapest, Hungary\\
email: emarci@renyi.hu\\
{\tt www.renyi.hu/\hbox{$\sim$}emarci}}

\begin{document}

\maketitle

\begin{abstract}

Cain, Clark and Rose defined a function $f\colon \RR^n \to \RR$ to
be \emph{vertically rigid} if $\graph(cf)$ is isometric to $\graph
(f)$ for every $c \neq 0$.
It is \emph{horizontally rigid} if $\graph(f(c \vec{x}))$
is isometric to $\graph (f)$ for every $c \neq 0$ (see \cite{CCR}).

In \cite{BE} the authors of the present paper settled Jankovi\'c's conjecture by showing that a continuous
function of one variable is vertically rigid if and only if it is of the form
$a+bx$ or $a+be^{kx}$ ($a,b,k \in \RR$). Later they proved in \cite{BE2} that a continuous
function of two variables is vertically rigid if and only if after a suitable
rotation around the $z$-axis it is of the form
$a + bx + dy$, $a + s(y)e^{kx}$ or $a + be^{kx} + dy$ ($a,b,d,k \in \RR$, $k
\neq 0$, $s : \RR \to \RR$ continuous). The problem remained open in higher
dimensions.

The characterization in the case of horizontal rigidity is
surprisingly simpler. C. Richter proved in \cite{Ri} that a
continuous function of one variable is horizontally rigid if and
only if it is of the form $a+bx$ ($a,b\in \RR$). The goal of the
present paper is to prove that a continuous function of two
variables is horizontally rigid if and only if it is of the form $a
+ bx + dy$ ($a,b,d \in \RR$). This problem also remains open in
higher dimensions.

The main new ingredient of the present paper is the use of functional equations.

\end{abstract}

\section{Introduction}

Cain, Clark and Rose introduced the notion of vertical
and horizontal rigidity in \cite{CCR}, which we now formulate for functions of
several variables.

\begin{definition} A function $f\colon \RR^n \to \RR$ is called
\emph{horizontally} or \emph{vertically rigid}, if $\graph(f(c
\cdot))$ or $\graph(cf)$ is isometric to $\graph (f)$ for all $c \in
(0,\infty)$, respectively. (Clearly, $c \in \RR \setminus \{0\}$
would be the same.)
\end{definition}

\begin{definition}  If $C$ is a subset of $(0, \infty)$ and $\iG$ is a
set of isometries of $\RR^{n+1}$ then we say that $f\colon
\RR^{n}\to \RR$ is horizontally or vertically rigid \emph{for a set
$C \subseteq (0, \infty)$ via elements of $\iG$} if for every $c \in
C$ there exists a $\varphi \in \iG$ such that $\varphi(\graph(f)) =
\graph (f(c \cdot))$ or $\varphi(\graph(f)) = \graph (cf)$,
respectively.

(If we do not mention $C$ or $\iG$ then $C$ is $(0, \infty)$ and
$\iG$ is the set of all isometries.)
\end{definition}

We characterized the continuous vertically rigid
functions of one variable in \cite{BE}, thus proving the conjecture of
D.~Jankovi\'c.

\begin{theorem}[Jankovi\'c's Conjecture] A continuous function $f\colon \RR \to \RR$ is vertically rigid if
and only if it is of the form $a+bx$ or $a+be^{kx}$ ($a,b,k \in
\RR$, $k \neq 0$).
\end{theorem}

C. Richter described the continuous horizontally rigid
functions in dimension one in \cite{Ri}.

\begin{theorem}[C. Richter] \label{1dim}
A continuous function $f\colon \RR \to \RR$ is horizontally rigid if
and only if it is of the form $a+bx$ ($a,b \in \RR$). (It is
actually enough to assume that $f$ is horizontally rigid for a set
$C \subseteq (0, \infty) \setminus \{1\}$ of at least $2$ elements.)
\end{theorem}

We characterized the continuous vertically rigid
functions of two variables in \cite{BE2}.

\begin{theorem}
A continuous function $f\colon \RR^2 \to \RR$ is vertically rigid if
and only if after a suitable rotation around the $z$-axis $f(x,y)$
is of the form $a + bx + dy$, $a + s(y) e^{kx}$ or $a + b e^{kx} +
dy$ ($a,b,d,k \in \RR$, $k \neq 0$, $s \colon \RR \to \RR$
continuous).
\end{theorem}

The goal of the present paper is to characterize the continuous
horizontally rigid functions of two variables.

\begin{theorem}[Main Theorem] A continuous function $f\colon \RR^2 \to \RR$ is
horizontally rigid if and only if it is of the form $a + bx + dy$
($a,b,d \in \RR$).
\end{theorem}

In fact, we will prove a slightly stronger statement for which we introduce the
following definition.

\begin{definition} Let us say that a set $C \subseteq (0, \infty)$ \emph{condensates
to $\infty$} if for every $r \in \RR$ the set $C \cap (r ,\infty)$
is uncountable.
\end{definition}

\begin{theorem}[Main Theorem, technical form] \label{main}
Let $C \subseteq (0, \infty)$ be a set condensating to $\infty$.
Then a continuous function $f\colon \RR^2 \to \RR$ is horizontally
rigid for $C$ if and only if it is of the form $a + bx + dy$ ($a,b,d
\in \RR$).
\end{theorem}

The proof will very heavily rely on techniques developed in \cite{BE2}. The key new ingredient will be that
functional equations enter the picture. More precisely, we will study a
specific system of functional equations in Section \ref{s:function}, which may be of interest
in its own right, and then apply it in the course of the proof of the Main Theorem.

\bigskip

Building heavily on \cite{BE2}, the structure of the proof will be as follows. In Section
\ref{s:function} we prove a theorem about the solutions of a
specific system of functional equations, which will be applied at the end of
the proof. In Section \ref{s:trans} we cite a theorem of C. Richter
which shows that if all the isometries are translations then the
continuous horizontally rigid function $f$ of arbitrary variables is
constant. We also characterize the horizontally rigid functions via
translations in arbitrary dimensions. In Section \ref{s:direction}
we consider the set $S_f$ of directions of segments connecting pairs
of points on $\graph(f)$ (see Definition \ref{d:S_f}). We cite a
theorem from \cite{BE2} to determine the possible $S_f$'s. In
Section \ref{s:end} we complete the proof by handling these cases
using various methods. Finally, in Section \ref{s:open} we collect
the open questions.

\section{A system of functional equations} \label{s:function}

In this section we explicitly solve a system of functional equations which may be interesting
in its own right. Our motivation is of course that we will apply this result for the proof of our Main Theorem.

\begin{theorem}
\label{function}
Suppose
that $C\subseteq (0,\infty)$ is an uncountable set and $u_c, v_c,
h_c \in \RR$ for all $c\in C$. Let $g\colon \RR \to \RR $ be a
continuous solution to the following system of functional equations.
\begin{equation} \label{eq:sys} g(x)=h_c g(cx+u_c)+v_c. \quad (x \in \RR, c \in C) \end{equation}
Then either $h_c = c^{-s}$ $(s \in \RR, s>0)$ for all $c\in C$ and
$g(x)=a+b_1 (d-x)^s$ for $x<d$ and $g(x)=a +b_2 (x-d)^{s}$ for
$x\geq d$ $(a,b_1,b_2,d \in \RR, (b_1,b_2)\neq (0,0))$ or $g$ is
constant.
\end{theorem}

\begin{proof}
Let us introduce the function $h(c) = h_c$ for the sake of
notational simplicity.

If $g$ is constant then the proof is complete, so we may assume that
this is not the case. If there exists a $c_0\in C$ such that
$h(c_0)=0$ then \eqref{eq:sys} implies $g(x)\equiv v_{c_0}$, but
this is impossible. Therefore $h(c)\neq 0$ for all $c\in C$.
Applying \eqref{eq:sys} twice for $c_1,c_2\in C$ we obtain the
following.
\begin{align*} g(x)&=h(c_1)g(c_1x+u_{c_1})+v_{c_1}  \\
&=h(c_1)\left(h(c_2)g\left(c_2(c_1x+u_{c_1})+u_{c_2}\right)+v_{c_2}\right)+v_{c_1}  \\
&=h(c_1)h(c_2)g(c_1 c_2x+c_2 u_{c_1}+u_{c_2})+h(c_1)v_{c_2}+v_{c_1}.
\end{align*}
Interchanging $c_1$ and $c_2$ follows
\begin{align*}
g(x) & =h(c_1)h(c_2)g(c_1 c_2x+c_2
u_{c_1}+u_{c_2})+h(c_1)v_{c_2}+v_{c_1} \notag  \\
& = h(c_1)h(c_2)g(c_1 c_2x+c_1 u_{c_2}+u_{c_1})+h(c_2)v_{c_1}+v_{c_2}.
\end{align*}
Therefore
\begin{align}
\label{eq1'}
&g(c_1 c_2x+c_2
u_{c_1}+u_{c_2}) +  \frac{h(c_1)v_{c_2}+v_{c_1}}{h(c_1)h(c_2)} \notag  \\
& = g(c_1 c_2x+c_1 u_{c_2}+u_{c_1}) + \frac{h(c_2)v_{c_1}+v_{c_2}}{h(c_1)h(c_2)}.
\end{align}
Let $y=c_1 c_2x+c_1 u_{c_2}+u_{c_1}$ be the argument of $g$ in the second line, and $u_{c_1,c_2}=u_{c_1}(c_2-1)-u_{c_2}(c_1-1)$,
then the argument of $g$ in the first line becomes $c_1
c_2x+c_2 u_{c_1}+u_{c_2}=y + u_{c_1,c_2}$.

Moreover, set $v_{c_1,c_2} = \frac{v_{c_1}(h(c_2)-1)-v_{c_2}(h(c_1)-1)}{h(c_1)h(c_2)}$, then \eqref{eq1'} attains the form
\[
g(y+u_{c_1,c_2})=g(y)+v_{c_1,c_2}.
\]
Since $y$ can also be an arbitrary real number, we can replace it by $x$ and hence we obtain that for every $x
 \in \RR$ and $c_1, c_2 \in C$
\begin{equation} \label{eq2} g(x+u_{c_1,c_2})=g(x)+v_{c_1,c_2}. \end{equation}
Define
\[
U=\{u_{c_1,c_2}: c_1,c_2\in C\}.
\]
First assume that $U$ is uncountable. Let us denote by $\widehat{U}$
the additive group generated by $U$, and let $G\colon \RR\to \RR$ be
defined by $G(x)=g(x)-g(0)$. Clearly, $G$ is continuous. If $n\in
\ZZ$ then applying equation \eqref{eq2} $|n|$ times follows that for
all $c_1,c_2\in C$
\begin{align} \label{x} G(x+nu_{c_1,c_2})&=g(x+nu_{c_1,c_2})-g(0)=g(x)+nv_{c_1,c_2}-g(0) \notag   \\
&=G(x)+nv_{c_1,c_2}.
\end{align}
Every $\widehat{u} \in \widehat{U}$ can be written as $\widehat{u}=
\sum _{i=1}^{k} n_i u_{c_{i},c'_{i}}$ ($k\in \NN,~n_i\in \ZZ,~
c_i,c'_{i}\in C,~ i\in \{1,\dots, k\}$). Then \eqref{x} and $G(0)=0$
imply
\begin{equation*} G\left( \sum
_{i=1}^{k} n_i u_{c_{i},c'_{i}}\right)=\sum_{i=1}^{k} n_i
v_{c_i,c'_i}.
\end{equation*}
Thus $G$ is additive on $\widehat{U}$. Since $U$ is uncountable,
$\widehat{U}$ is dense in $\RR$. Therefore the continuity of $G$
follows that $G$ is additive. Since $G$ is an additive and continuous function,
it is of the form $G(x)=bx$ $(b \in \RR)$. Hence the non-constant function $g$ is of the form $g(x)=a+bx$, where
$b\in \RR\setminus \{0\}$ and $a=g(0)\in \RR$. If $c\in C$ is fixed
then \eqref{eq:sys} implies $a+bx=h(c)(a+bcx+u_c)+v_c$. Comparing
the coefficients of $x$ on the two sides of the equation yields
$h(c)=c^{-1}$. Therefore $g$ and $h$ are of the required form.

Suppose next that $U$ is countable and $1\notin C$. Let $c_1,c_2\in
C$ be arbitrarily fixed. The map $F\colon C\to U\times U$,
$F(c)=(u_{c_1,c},u_{c_2,c})$ has an uncountable domain and a
countable range, so it is not one-to-one. Hence there exist
$c_3,c_4\in C$, $c_3\neq c_4$ such that $F(c_3)=F(c_4)$. Therefore
$u_{c_1,c_3}=u_{c_1,c_4}$ and $u_{c_2,c_3}=u_{c_2,c_4}$, so by the definition of $u_{c_1, c_2}$
\begin{align*}
&u_{c_1}(c_3-1)-u_{c_3}(c_1-1)=u_{c_1}(c_4-1)-u_{c_4}(c_1-1), \\
&u_{c_2}(c_3-1)-u_{c_3}(c_2-1)=u_{c_2}(c_4-1)-u_{c_4}(c_2-1).
\end{align*}
Thus
\begin{align*}
u_{c_1} (c_3 - c_4) = (u_{c_3} - u_{c_4}) (c_1 - 1), \\
u_{c_2} (c_3 - c_4) = (u_{c_3} - u_{c_4}) (c_2 - 1),\\
\end{align*}
and dividing by $(c_3 - c_4) (c_1 - 1)$ or $(c_3 - c_4) (c_2 -
1)$, respectively infers
$$\frac{u_{c_1}}{c_1-1}=\frac{u_{c_3}-u_{c_4}}{c_3-c_4}=\frac{u_{c_2}}{c_2-1}.$$
As $c_1, c_2 \in C$ were arbitrary, there exists $d \in \RR$ such that $\frac{u_{c}}{c-1}=-d$ for
all $c\in C$, that is, $u_c=d (1-c)$.

Substituting this into \eqref{eq:sys} we obtain $g(x)=h(c)g(cx+d
(1-c))+v_c$. Let $g_{d}(x)=g(x+d)$, so we have $g_{d}(x-d)=h(c)
g_{d}\left(c(x-d)\right) + v_c$. Putting $x$ for $x-d$ implies
\begin{equation}
\label{e:5} g_{d}(x) = h(c)g_{d}(cx) + v_{c} \ (x\in \RR, c\in
C).
\end{equation}
Plugging in $x=0$ into this equation we obtain
$v_c=g_{d}(0)(1-h(c))$, and then writing this back to \eqref{e:5} yields
$g_{d}(x)-g_{d}(0)=h(c)\left(g_{d}(cx)-g_{d}(0)\right)$. Let
$f\colon \RR\to \RR$ be defined by $f(x)=g_{d}(x)-g_{d}(0)$. Clearly, $f$ is
continuous, and for all $c\in C$ and $x\in \RR$
\begin{equation} \label{eq3}
f(x)=h(c)f(cx).
\end{equation}
If $n\in \ZZ$ then applying \eqref{eq3} $|n|$ times follows that for
all $c\in C$ and $x\in \RR$ we have
\begin{equation} \label{4x} f(x)=h^{n}(c)f(c^{n}x). \end{equation}

Let us now examine $g|_{(-\infty, d)}$.

\noindent\textbf{Case I.} $f|_{(-\infty, 0)}\equiv 0$.

This holds iff $g_{d}|_{(-\infty, 0)}\equiv g_{d}(0)$ iff
$g|_{(-\infty, d)}\equiv g(d)$ iff $g|_{(-\infty, d)}$ is constant
(note that $g$ is continuous at $d$). Hence $g|_{(-\infty, d)}$ is
of the required form with $a = g(d), b_1 = 0$ and $s \in \RR$
arbitrary.  This clearly extends to $x=d$ as well.

\noindent\textbf{Case II.} There is an $x_0 < 0$ such that
$f(x_0)\neq 0$, that is,  $g|_{(-\infty, d)}$ is not constant.

Let us denote by $\widetilde{C}$ the multiplicative
group generated by $C$, and let $H\colon (0,\infty) \to \RR$ be defined by
$H(y)=\frac{f(yx_0)}{f(x_0)}$. Then $H$ is continuous and
\eqref{eq3} follows that $H(c)=\frac{1}{h(c)}$ for all $c\in C$.
Every $\widetilde{c} \in \widetilde{C}$ can be written as
$\widetilde{c}= \prod _{i=1}^{k} c_{i}^{n_i}$ ($k\in \NN,~n_i\in
\ZZ,~ c_i\in C,~ i\in \{1,\dots, k\}$). Then \eqref{4x} implies
\begin{align*} H\left(\prod _{i=1}^{k}
c_{i}^{n_i}\right)&=\frac{f\left(c_1^{n_1}\cdots
c_{k}^{n_{k}}x_0\right)}{f(x_0)}=\prod_{i=1}^{k} \frac{f\left(
c_1^{n_1} \cdots c_{i}^{n_i}x_0\right)}{f\left(c_1^{n_1}\cdots
c_{i-1}^{n_{i-1}}x_0\right)} \\
&= \prod _{i=1}^{k} h^{-n_i}(c_{i}),
\end{align*}
thus $H$ is multiplicative on $\widetilde{C}$. Since $C$ is
uncountable, $\widetilde{C}$ is dense in $(0,\infty)$. Therefore the
continuity of $H$ follows that $H$ is multiplicative. Since $H$ is a
non-zero continuous multiplicative function, it is of the form
$H(y)=y^s$ $(s\in \RR)$. Then, on the one hand,
$h(c)=\frac{1}{H(c)}=c^{-s}$ for all $c\in C$, so $h$ is of the
required form (later we will prove $s>0$). On the other hand, for an
arbitrary $x < 0$ letting $y = \frac{x}{x_0} > 0$ yields
\begin{align*}
g_d(x) - g_d(0) &= f(x) = f(y x_0) = f(x_0)H(y) \\
&=f(x_0)y^s = f(x_0) \left( \frac{x}{x_0} \right)^s = f(x_0)
\frac{|x|^s}{|x_0|^s}.
\end{align*}
Hence, applying the definition of $g_d$ and writing $x-d$ for
$x$ gives that for some $s \in \RR$
\begin{equation}
\label{e:g} g(x) = g(d) + \frac{f(x_0)}{|x_0|^s} |x-d|^s  \quad
(x<d),
\end{equation}
which is the required form. 
The continuity of $g$ at $d$ implies
that $s \geq 0$. Since $g|_{(-\infty, d)}$ is
not constant, we also have $b_1 \neq 0$ and $s > 0$. Then \eqref{e:g} holds for $x=d$ as well.

An analogous argument works for $x \ge d$. Since $g$ is not
constant, Case II holds either for $x \le d$ or for $x \ge d$.
Therefore  there exists an $s>0$ such that $h(c) = c^{-s}$ for every
$c \in C$, and $g$ is of the required form with $(b_1, b_2) \neq
(0,0)$. (Note that $s$ is determined by $h$, so it does not depend
on $x \le d$ or $x \ge d$, and similarly for $a$, since $a = g(d)$.)
This finishes the proof for $U$ countable and $1\in C$.

Finally, if $U$ is countable and $1\in C$ then the above argument
for $C\setminus \{1\}$ implies that $g,h|_{C\setminus \{1\}}$ are of
the required forms, so it is enough to prove that $h(1)=1$. Since
$(b_1,b_2)\neq (0,0)$, we may assume by the symmetry that $b_2\neq
0$. Applying \eqref{eq:sys} for $c=1$ and $x>\max\{d,d-u_1\}$ we
obtain $a+b_2(x-d)^s=h(1)\left(a+b_2(x+u_1-d)^s\right)+v_1$. That is
$h(1)=\lim \limits_{x\to \infty}
\frac{a-v_1+b_2(x-d)^s}{a+b_2(x+u_1-d)^s}=1$. The proof is complete.
\end{proof}

\begin{remark}
From the above theorem it is easy to get all the solutions of the system of functional equations in question,
but we did not write it out that way, since it would have made it very difficult to read.
\end{remark}

\section{Horizontal rigidity via translations} \label{s:trans}

\begin{theorem}[C.~Richter] \label{trans} Let $c \in (0, \infty) \setminus \{1\}$
be arbitrary, and let $f : \RR^n \to \RR$ be a continuous function
that is horizontally rigid for $c$ via a translation. Then $f$ is
constant.
\end{theorem}

\begin{proof} The proof of \cite[Prop. 3]{Ri} works with the obvious
modifications, just replace $x$ and $u$ by vectors.
\end{proof}

The following theorem generalizes the one-dimensional result
\cite[Thm. 6.2]{BE}. It will not be used in the sequel, but it is
interesting in its own right. Note that we do not assume continuity.

\begin{theorem}  A function $f \colon \RR^n \to \RR$ is horizontally
rigid via translations if and only if there exists a $\vec{r} \in
\RR^n$ so that $f$ is constant on every open halfline starting from
$\vec{r}$.
\end{theorem}

\begin{proof}[Sketch of the proof] If $\vec{r} \in
\RR^n$ and $f$ is constant on every open halfline starting from
$\vec{r}$ then $f$ is clearly horizontally rigid via
translations.

In order to prove the nontrivial direction, we use induction on $n$. For $n=1$ the
statement is \cite[Thm. 6.2]{BE}. Suppose that the statement is true
for $n-1$.

First, assume that $f$ is not periodic. Then the proof of \cite[Thm. 6.2]{BE} applies with the
obvious modifications, we just replace $x,u,p$ and $r$ by vectors.

Finally, assume that $f$ is periodic modulo $\vec{p}\in \RR^{n}$
for some $\vec{p}\neq \vec{0}$.  As in the proof of \cite[Thm.
6.2]{BE}, we obtain that $f$ is also periodic modulo $c\vec{p}$ for
every $c\in (0,\infty)$.

Let $\varphi$ be an isometry of $\RR^n$. It is not hard to see that
if $f$ is horizontally rigid via translations then so is $f \circ
\varphi$. Similarly, if $f$ is constant on every open halfline
starting from some $\vec{r}$ then so is $f \circ \varphi$. Therefore
we may assume that $f$ is periodic modulo every vector of the form
$(0,\dots,0,x_{n+1})$. In other words, $f$ is of the form $g\circ
\pr$, where $g\colon \RR^{n-1} \to \RR$ is a function and
$\pr((x_1,\dots,x_n))=(x_1,\dots,x_{n-1})$ is the projection of
$\RR^{n}$ onto the first $n-1$ coordinates. Since $f$ is
horizontally rigid via translations, for all $c\in (0,\infty)$ there
are vectors $\vec{u}_c,\vec{v}_c\in \RR^{n}$ such that
$f(c\vec{x})=f(\vec{x}+\vec{u}_c)+\vec{v}_c$. Then clearly for all
$\vec{y}\in \RR^{n-1}$ we have $g(c\vec{y})=g(\vec{y}+
\pr(\vec{u}_c))+\pr(\vec{v}_c)$, so $g$ is also horizontally rigid
via translations. By the inductional hypothesis $g$ is of the
required form. Thus $f=g\circ \pr$ is also of the required form, and
the proof is complete.
\end{proof}

\section{Determining the possible set of directions}
\label{s:direction}

Now we start working on the case of general isometries. We follow \cite{BE2}.

Let $\SS^2 \subseteq \RR^3$ denote the unit sphere. For a function
$f\colon \RR^2 \to \RR$ let $S_f$ be the set of directions between
pairs of points on the graph of $f$, that is,

\begin{definition} \label{d:S_f}
$$S_f = \left\{ \frac{p-q}{|p-q|} \in \SS^2 : p,q \in \graph(f),\  p
\neq q \right\}.$$
\end{definition}

Recall that a \emph{great circle} is a circle line in $\RR^3$ of
radius $1$ centered at the origin. We call it \emph{vertical} if it
passes through the points $(0,0,\pm 1)$.

\begin{definition} \label{d:h} Let $h : \SS^1 \to \SS^2$ be defined as follows.
Every $\vec{x} \in \SS^1$ is in a unique half great circle
connecting $(0,0,1)$ and $(0,0,-1)$. It is not hard to see using the continuity of $f$ that
the intersection of $S_f$ with this great circle is an arc. Define $h(\vec{x})$ as the top endpoint
of this arc.
\end{definition}

Clearly, the bottom endpoint of this arc is $-h(-\vec{x})$, so the
`top function bounding the strip $S_f$ is $h(\vec{x})$ and the
bottom function is $-h(-\vec{x})$'. The coordinate functions of $h$
are denoted by $(h_1,h_2,h_3)$, where $h_3 : \SS^1 \to [-1, 1]$
encodes all information about $h$.

\begin{definition} \label{d:psi} For $c > 0$ let $\psi_c: \SS^2 \to \SS^2$ denote the map that `deforms $S_f$
according to the map $c \mapsto f(c \cdot )$', that is,
$$\psi_c((x,y,z)) =\frac{\left(\frac{x}{c},\frac{y}{c},z\right)}{|\left(\frac{x}{c},\frac{y}{c},z\right)|}
= \frac{(x,y,cz)}{|(x,y,cz)|}
 \ \ \left((x,y,z) \in \SS^2\right).$$
\end{definition}

The above equation shows that the map that `deforms $S_f$ according
to the map $c \mapsto cf$' is also $\psi_c$. This is the key
connection between vertical and horizontal rigidity. This explains
why we can replace vertically rigid functions by
horizontally rigid functions in \cite[Thm. 5.3]{BE2} and obtain the
following theorem by the same proof.

\begin{theorem} \label{t:cases} Let $C \subseteq (0, \infty)$ be a set condensating
to $\infty$, and let $f\colon \RR^2 \to \RR$ be a continuous
function horizontally rigid for $C$. Then one of the
following holds.
\begin{itemize}

\item
\textbf{Case A.} There is a vertical great circle that intersects
$S_f$ in only two points.

\item
\textbf{Case B.} $S_f = \SS^2 \setminus \{ (0,0,1), (0,0,-1) \}$.

\item
\textbf{Case C.} There exists an $\vec{x_0} \in \SS^1$ such that
$h_3(\vec{x_0}) = 0$ and $h_3(\vec{x}) = 1$ for every $\vec{x} \neq
\vec{x_0}$, that is, $S_f$ is `$\SS^2$ minus two quarters of a great
circle'.

\item
\textbf{Case D.} There exists a closed interval $I$ in $\SS^1$ with
$0 < \length(I) < \pi$ such that $h_3(\vec{x}) = 0$ if $\vec{x} \in
I$, and $h_3(\vec{x}) = 1$ if $\vec{x} \notin I$, that is, $S_f$ is
`$\SS^2$ minus two spherical triangles'.

\end{itemize}
\end{theorem}

\section{The end of the proof} \label{s:end}

This section builds heavily on the corresponding section of \cite{BE2}.

The following lemma will be useful, the easy proof is left to the
reader.

\begin{lemma} \label{easy} Let $f\colon \RR^2 \to \RR$ be horizontally rigid for
$c_0$ via $\varphi_0$ and for $c$ via $\varphi$. Then $f(c_0 \cdot)$
is horizontally rigid for $\frac{c}{c_0}$ via $\varphi\circ
\varphi_0^{-1}$.
\end{lemma}

The following technical lemma is \cite[Lemma 2.1]{BE2}.

\begin{lemma} \label{rotation} Let $f(x,y)=g(x)+dy$, where $d>0$ and let $c>0$. If
we rotate $\graph(f)$ around the $x$-axis by the angle
$\alpha_c=\arctan (cd)-\arctan(d)$ then the intersection of this
rotated graph with the $xy$-plane is the graph of a function of the
form $y =- w_{c,d} g(x) $, where
$w_{c,d}=\frac{1}{cd}\sqrt{\frac{(cd)^2+1}{d^{2}+1}}$.
\end{lemma}

\begin{remark} \label{r:rigid} Let $\varphi_c$ be the isometry mapping
$\graph(f)$ onto $\graph(f(c \cdot))$. Every isometry $\varphi$ is
of the form $\varphi^{trans} \circ \varphi^{ort}$, where
$\varphi^{ort}$ is an orthogonal transformation and
$\varphi^{trans}$ is a translation. Moreover, if $\varphi$ is
orientation-preserving then $\varphi^{ort}$ is a rotation around a
line passing through the origin. A key observation is the following:
The horizontal rigidity of $f$ for $C$ implies that $\psi_c(S_f)
=\varphi_c^{ort}(S_f)$ for every $c \in C$.
\end{remark}

Now we complete the proof of the technical form of the Main Theorem.
We repeat the statement here.

\begin{theorem}[Main Theorem, technical form] Let $C \subseteq (0, \infty)$
be a set condensating to $\infty$.
Then a continuous function $f\colon \RR^2 \to \RR$ is horizontally
rigid for $C$ if and only if it is of the form $a + bx + dy$ ($a,b,d
\in \RR$).
\end{theorem}

\begin{proof} By Theorem \ref{t:cases} it suffices to consider Cases A-D.

\noindent\textbf{Case A.} There is a vertical great circle that
intersects $S_f$ in only two points.

We may assume using a suitable rotation around the $z$-axis that the
vertical great circle is in the $yz$-plane, hence $f(x,y)$ is of the form
$g(x) + dy$. (Note that both the class of affine functions and the class
of functions that are horizontally rigid for $C$ are closed under
rotations around the $z$-axis.) The continuity of $f$ implies that $g$
is also continuous.

\textbf{Subcase A1.} $d = 0$.

Let $c \in C$ be fixed, and let $\varphi_c$ be the corresponding
isometry. The graph of $f(c \cdot)$ is invariant under translations
parallel to the $y$-axis. As the same holds for $f$, by rigidity,
$f(c \cdot)$ is also invariant under translations parallel to the
$\varphi_c$-image of the $y$-axis. If these two directions are
nonparallel, then $\graph(f(c \cdot))$ is a plane, and hence so is
$\graph(f)$.
Therefore we may assume that all lines parallel to the $y$-axis are
taken to lines parallel to the $y$-axis, but then all planes
parallel to the $xz$-plane are taken to planes parallel to the
$xz$-plane. But this shows (by considering the intersections of the
graphs with the $xz$-plane) that $g$ is horizontally rigid for $c$,
hence $g(x)$ is of the from $a + bx$ ($a,b \in
\RR$) by Theorem \ref{1dim}, and we are done.

\textbf{Subcase A2.} $d \neq 0$.

We may assume that $d>0$, since otherwise we may consider $-f$.

For every $c \in C$ let $\varphi_c$ be the corresponding isometry.
We claim that we may assume that all these are
orientation-preserving. First, if
$\{c \in C : \varphi_c \textrm{ is orientation-preserving} \}$
condensates to $\infty$ then we are done by shrinking $C$. Otherwise,
we may assume that they are all orientation-reversing (note that if
we split $C$ into two pieces then at least one of them still
condensates to $\infty$). Let us fix a $c_0 \in C$ and consider
$c_0f$ instead of $f$. By Lemma \ref{easy} this function is rigid
for an uncountable set with all isometries orientation-preserving,
and if it is of the desired form then so is $f$, so we are done.

We may assume $1 \notin C$. Let us fix a $c \in C$. Similarly as in
the previous subcase, we may assume that lines parallel to $(0,1,d)$
are taken to lines parallel to $(0,1,cd)$ as follows. The special
form of $f$ implies that $\graph(f)$ is invariant under translations
in the $(0,1,d)$-direction, hence $\graph(f(c \cdot))$ is invariant
under translations in the $(0,1,cd)$-direction, moreover, by
rigidity, $\graph(f(c \cdot))$ is also invariant under translations
parallel to the $\varphi_c$-image of the lines of direction
$(0,1,d)$. If these two latter directions do not coincide then
$\graph(f(c \cdot))$ is a plane, and we are done.

Therefore the image of every line parallel to $(0,1,d)$ is a line
parallel to $(0,1,cd)$ under the orientation-preserving isometry
$\varphi_c$. As in Remark \ref{r:rigid}, write $\varphi_c =
\varphi_c^{trans} \circ \varphi_c^{ort}$, where $\varphi_c^{ort}$ is
a rotation about a line containing the origin and
$\varphi_c^{trans}$ is a translation. Since the translation does not
affect directions, the rotation $\varphi_c^{ort}$ takes the
direction $(0,1,d)$ to the nonparallel direction $(0,1,cd)$ ($d \neq
0$), therefore the axis of the rotation has to be orthogonal to the
plane spanned by these two directions. Hence the axis has to be the
$x$-axis. Moreover, the angle of the rotation is easily seen to be
$\arctan(cd) - \arctan(d)$.

We now show that we may assume that $\varphi_c^{trans}$ is a
horizontal translation. Decompose the translation as
$\varphi_c^{trans} = \varphi_c^{\vec{u}} \circ \varphi_c^{\vec{v}}$,
where $\varphi_c^{\vec{v}}$ is a horizontal translation and
$\varphi_c^{\vec{u}}$ is a translation in the $(0,1,cd)$-direction.
Since $\varphi_c^{ort} (\graph(f))$ is invariant under translations
in the $(0,1,cd)$-direction, so is $\varphi_c^{\vec{v}} \circ
\varphi_c^{ort} (\graph(f))$, hence
\begin{align*} \varphi_c^{\vec{v}} \circ \varphi_c^{ort} (\graph(f)) &=
\varphi_c^{\vec{u}} \circ \varphi_c^{\vec{v}} \circ \varphi_c^{ort}
(\graph(f)) \\
&= \varphi_c (\graph(f)) = \graph(f(c\cdot)),
\end{align*}
so we can assume $\varphi_c = \varphi_c^{\vec{v}} \circ
\varphi_c^{ort}$, and we are done.

Let us denote the $xy$-plane by $\{z=0\}$ and consider the
intersection of both sides of the equation $\varphi_c (\graph(f)) =
\graph(f(c  \cdot))$ with $\{z=0\}$. On the one hand,
\begin{align*} \{z=0\} \cap \varphi_c \left(\graph(f)\right) &= \{z=0\} \cap
\varphi_c^{\vec{v}} \circ \varphi_c^{ort} \left(\graph(f)\right) \\
&=\varphi_c^{\vec{v}} \left( \{z=0\} \cap \varphi_c^{ort} \left(\graph(f)\right)\right)\\
&=\varphi_c^{\vec{v}} \left( \graph \left( -w_{c,d} g
\right)\right),
\end{align*}
where we used the fact that $\varphi_c^{\vec{v}}$ is horizontal and
Lemma \ref{rotation}. On the other hand, it is easy to see that
$$\{z=0\} \cap \graph(f(c \cdot)) = \graph\left( -\frac{1}{cd} g(c
\cdot) \right).$$
If $\vec{v}=(v_{c,1}, v_{c,2})$ then we obtain for all $c\in C$
and $x\in \RR$
$$-w_{c,d} g(x-v_{c,1})+v_{c,2}=-\frac{1}{cd}g(cx).$$
For every $c \in C$ let $h_c= \frac{1}{cd
w_{c,d}}=\sqrt{\frac{d^2+1}{(cd)^2+1}}$, $u_c=cv_{c,1}$ and
$v_c=\frac{v_{c,2}}{w_{c,d}}$. Reordering the above functional
equations and substituting $x+v_{c,1}$ for $x$ yield for all $c\in
C$ and $x\in \RR$
$$g(x)=h_c g(cx+u_c)+v_c.$$
Then Theorem \ref{function} implies that either $g$ is constant or
$h_c = c^{-s}$ for every $c\in C$. But the latter cannot
hold, since $\sqrt{\frac{d^2+1}{(xd)^2+1}}$ and $x^{-s}$ are
clearly nonidentical analytic functions, so they cannot coincide on the
uncountable set $C$. Hence $g$ is constant, and the proof of Case A
is complete.

\noindent\textbf{Case B.} $S_f = \SS^2 \setminus\{(0,0,1),
(0,0,-1)\}$.

$S_f$ is invariant under every $\psi_c$, and hence so is under every
$\varphi_c^{ort}$. Then clearly $\varphi_c^{ort}((0,0,1)) = (0,0,1)$
or $\varphi_c^{ort}((0,0,1)) = (0,0,-1)$ for every $c \in C$. By the
same argument as above, we can assume that the former holds for
every $c \in C$. Using the argument again we can assume that all
$\varphi_c$'s are orientation-preserving. But then each of these is
a rotation around the $z$-axis followed by a translation, in other
words, an orientation-preserving transformation in the $xy$-plane
followed by a translation in the $z$-direction. An
orientation-preserving transformation in the plane is either a
translation or a rotation. If it is a translation for every $c$ then
we are done by Theorem \ref{trans}. So let us assume that there
exists a $c$ such that $\varphi_c$ is a proper rotation around
$\vec{x} \in \RR^2$ followed by a vertical translation. We claim
that then $f$ is constant, which will contradict that $S_f$ is
nearly the full sphere, finishing the proof of this case. We may
assume that $c>1$, since otherwise we consider $f(c \cdot)$ instead
of $f$ and replace $c$ and $\varphi_c$ by $1/c$ and
$\varphi_c^{-1}$, respectively. Let $R>0$ be so large such that
$B(\vec{x}/c,R/c) \subseteq B(\vec{x},R)$. It is enough to show that
$f$ is constant on $B(\vec{x},R)$. Assume on the contrary that this
is not the case, and let $\delta>0$ be the distance between the
compact sets where $f$ attains its minimum and maximum on
$B(\vec{x},R)$. Choose one point of the two sets each,
$\vec{x}_{min}$ and $\vec{x}_{max}$, such that
$\delta=|\vec{x}_{min}-\vec{x}_{max}|$. On the one hand, the distance
between the sets where the minimum and maximum are attained should still
be $\delta$ for $f(c \cdot)$, since it is not affected by rotations
around $\vec{x}$ or vertical translations. But on the other hand,
the function $f(c \cdot)$ takes the same values in $\vec{x}_{min}/c$
and $\vec{x}_{max}/c$ as $f$ in $\vec{x}_{min}$ and $\vec{x}_{max}$,
respectively. Therefore $\vec{x}_{min}/c,\vec{x}_{max}/c\in
B(\vec{x}/c,R/c)\subseteq B(\vec{x},R)$. Thus the distance between
the sets where $f(c \cdot)$ attains $\min_{B(\vec{x},R)} f$ and
$\max_{B(\vec{x},R)} f$ is at most
$|\vec{x}_{min}/c-\vec{x}_{max}/c|=\delta/c<\delta$, a
contradiction.

\noindent\textbf{Case C.} There exists an $\vec{x_0} \in \SS^1$ such
that $h_3(\vec{x_0}) = 0$ and $h_3(\vec{x}) = 1$ for every $\vec{x}
\neq \vec{x_0}$, that is, $S_f$ is `$\SS^2$ minus two quarters of a
great circle'.

Thus $S_f$ is invariant under every $\psi_c$, and hence so is
under every $\varphi_c^{ort}$. Hence $\varphi_c^{ort}$ maps
$(0,0,1)$ to one of the four endpoints of the two arcs. Therefore we
can assume by splitting $C$ into four pieces according to the image
of $(0,0,1)$ and applying Lemma \ref{easy} that $(0,0,1)$ is a fixed
point of every $\varphi_c^{ort}$. But then the two arcs are also
fixed, and actually $\varphi_c^{ort}$ is the identity. Hence every
$\varphi_c$ is a translation, and we are done by Theorem
\ref{trans}.

\noindent\textbf{Case D.} There exists a closed interval $I$ in
$\SS^1$ with $0 < \length(I) < \pi$ such that $h_3(\vec{x}) = 0$ if
$\vec{x} \in I$ and $h_3(\vec{x}) = 1$ if $\vec{x} \notin I$, that
is, $S_f$ is `$\SS^2$ minus two spherical triangles'.

As $S_f$ is invariant under every $\varphi_c^{ort}$, vertices of the
triangles are mapped to vertices. Hence we may assume (by splitting
$C$ into six pieces) that $(0,0,1)$ is fixed. But then the triangles
are also fixed sets, and every $\varphi_c^{ort}$ is the identity, so
we are done as in the previous case.

This finishes the proof of the Main Theorem.
\end{proof}

\section{Open questions} \label{s:open}

All three questions below are somewhat analogous to the ones
asked in \cite{BE2}.

\begin{question}
Which notion of largeness of $C$ suffices for the Main Theorem? For
example, does the Main Theorem hold if we only assume that $C$
is uncountable or $C$ generates a dense subgroup of $(0,\infty)$ or
$C$ has at least two elements different from 1?
\end{question}

\begin{question}
Let us call a set $H \subseteq \SS^2$ rigid if $\psi_c(H)$ is
isometric to $H$ for every $c > 0$. Is there a simple description of
rigid sets? Or if we assume some regularity?
\end{question}

And finally, the most intriguing problem.

\begin{question}
What can we say if there are more than two variables? Is every
continuous horizontally rigid function affine?
\end{question}

\end{document}